\documentclass[sn-mathphys]{sn-jnl}% Math and Physical Sciences Reference Style
 \usepackage{amsmath, amsthm, amsfonts}
 \usepackage{mathrsfs, amssymb}
 \usepackage{graphicx, color, bm}
\jyear{2021}%
\theoremstyle{thmstyleone}%
 \newtheorem{thm}{Theorem}[section]
 \newtheorem{cor}[thm]{Corollary}
 \newtheorem{lem}[thm]{Lemma}
 \newtheorem{prop}[thm]{Proposition}
 \newtheorem{cond}[thm]{Condition}

\theoremstyle{thmstyletwo}%
 \theoremstyle{remark}
 
 \theoremstyle{example}

\theoremstyle{thmstylethree}%

 \def\ar{\!\!\!&}\def\nnm{\nonumber}

 \def\mbb{\mathbb}\def\mbf{\mathbf}
 \def\mrm{\mathrm}

 \def\beqlb{\begin{eqnarray}}\def\eeqlb{\end{eqnarray}}
 \def\beqnn{\begin{eqnarray*}}\def\eeqnn{\end{eqnarray*}}
 \def\d{{\mbox{\rm d}}}\def\e{{\mbox{\rm e}}}\def\p{\mathbf{P}}

 \def\eqref#1{{\rm(\ref{#1})}}

\begin{document}

\title[Wasserstein-type distances of two-type CBIRE-processes]{Wasserstein-type distances of two-type continuous-state branching processes in L\'{e}vy random environments}


\author[1]{\fnm{Chen} \sur{Shukai}}\email{skchen@fjnu.edu.cn}
\equalcont{These authors contributed equally to this work.}
\author[1]{\fnm{Fang} \sur{Rongjuan}}\email{fangrj@fjnu.edu.cn}
\equalcont{These authors contributed equally to this work.}
\author*[2]{\fnm{Zheng} \sur{Xiangqi}}\email{zhengxq@ecust.edu.cn}
\equalcont{These authors contributed equally to this work.}

\affil[1]{\orgdiv{College of Mathematics and Informatics}, \orgname{Fujian Normal University}, \orgaddress{ \city{Fuzhou}, \postcode{350007}, \country{People's Republic of China}}}

%\affil[2]{\orgdiv{Department}, \orgname{Organization}, \orgaddress{\street{Street}, \city{City}, \postcode{10587}, \state{State}, \country{Country}}}

\affil*[2]{\orgdiv{School of Mathematics}, \orgname{East China University of Science and Technology}, \orgaddress{ \city{Shanghai}, \postcode{200237}, \country{People's Republic of China}}}


\abstract{Under natural conditions, we proved the exponential ergodicity in Wasserstein distance of two-type continuous-state branching processes in L\'evy random environments with immigration. Furthermore, we expressed accurately the parameters of the exponent. The coupling method and the conditioned branching property play an important role in the approach. Using the tool of superprocesses, the ergodicity in total variance distance is also proved.}




\maketitle

\section{Introduction}\label{sec1}

The model of two-type continuous-state branching processes with immigration in L\'evy random environments(CBIRE-processes) was established by Qin and Zheng in \cite{qinzheng}. And the authors provided an equivalent condition for ergodicity in the $ L_1$-Wasserstein distance. Continuing the work, we are concerned in this paper with the convergence rate of the ergodicity in the $ L_1$-Wasserstein distance and the ergodicity in the total variance distance.

CBIRE-processes are derived from the model of classical continuous-state branching processes with immigration(CBI-processes). The ergodic property is an important topic in the study of CBI-processes. It was proved in \cite{li} that, for a subcritical CBI-process with branching mechanism $\phi$ strictly positive on $(0,\infty),$ the ergodicity holds if and only if
$$\int_0^z\frac{\psi(\lambda)}{\phi(\lambda)}\d \lambda<\infty,\quad \text{for\ some}\ z>0.$$
Here, $\psi$ is the immigration mechanism defined by
 \beqlb\label{psi}
\psi(\lambda) = h\lambda + \int_0^\infty ({1-\e^{-z\lambda}})n(\d z),
 \eeqlb
where $h\ge 0$ and $(1\wedge z) n(\d z)$ is a finite measure on $(0,\infty).$
This is equivalent to $\int_1^{\infty}\log(u)n(\d u)<\infty$ in the strictly subcritical case. Moreover, \cite{lima} proved the exponential ergodicity when the process is strictly subcritical and the immigration mechanism $\psi(\lambda)$ takes the particular form of $a\lambda,$ where $a$ is a positive constant. In fact, for a more general form of $\psi(\lambda)$ as \eqref{psi}, the conclusion is still valid, see pp.66-67 of \cite{li1}. And the proof used the method of coupling, which
 has been proved effective in the study of exponential ergodicity, see also \cite{wangfy,liwang,li2, wangj} for instance.

For processes in random environments, limited work has been done in the topic of ergodicity. The model of CBIRE-processes was established in 2018 by \cite{helixu}, also \cite{pardo}. For a CBIRE-process with branching mechanism $\phi,$  immigration mechanism $\psi$ and random environment $\xi,$ where $\{\xi(t):t\ge 0\}$ is a L\'{e}vy process, \cite{helixu} proved that, under some hypothesis on $\{\xi(t):t\ge 0\},$  the equivalent condition of ergodicity is $$\int_1^{\infty}\log(u)n(\d u)<\infty.$$
Recently, \cite{xu} studied its polynomial ergodicity when the branching mechanism is $\alpha$-stable in the distance of total variation. Furthermore, \cite{jin1} provided a sufficient condition for exponential ergodicity of CBIRE-processes in both the Wasserstein distance and the total variance distance. The above results are concerned with the one-dimensional case, whereas in this paper, we consider similar problems in the two-dimensional setting.

Our main results consist of two parts, the exponential ergodicity in the $ L_1$-Wasserstein distance and the ergodicity in the total variance distance. When proving the exponential ergodicity of two-type CBIRE-processes, the main difficulties lie in dealing with the random environment: because of the random environment, it is more challenging to construct the test function as in \cite{liwang}; and unlike the processes in \cite{li2}, the two-type CBIRE-processes no longer hold the branching property. Fortunately, owing to the conditioned branching property of the two-type CBIRE-processes, it is still possible to use a coupling similar to \cite{li2} after some adjustments. Moreover, lightened by the skew convolution semigroup of CBI-processes on p.66 of \cite{li}, we introduce a random skew convolution semigroup when constructing the processes with immigration in L\'evy random environments, which we have not seen in previous researches. When proving the ergodicity in the total variance distance, the absence of the corresponding Grey's condition in multi-dimensional cases is the most critical difficulty. We conquer it by linking the multi-dimensional cases to the single-dimensional case with the help of different spatial motions in the setting of Dawson-Watanabe superprocesses. Coincidentally, in the very recent work, the Grey's condition for multi-type CBI-processes was established in \cite{cha} by using the Lamperti representation instead of our tool of superprocesses. But the main idea of constructing the local projection is similar.

In the next section, we present our main results after some necessary reviews of the two-type CBIRE-processes and some basic knowledge on the Wasserstein distance. Theorem 2.1 gives a sufficient condition for the exponential ergodicity in $ L_1$-Wasserstein distance, and its exact expression of the parameters is provided in its proof. Theorem 2.2 provides a sufficient condition for the exponential ergodicity in the total variance distance. Section 3 and Section 4 are devoted to the proofs of  Theorem 2.1 and Theorem 2.2. Our method is strongly influenced by \cite{li2}. In the proofs, we make full use of the the conditioned branching property. 

%%%%%%%%%%%%%%%%%%%%%%%%%%%%%%%%%%%%%%%%%%%%%%%%%%%%%%%%%%%%%%%%%%%%%%%%%%%%%%%%%%%%%%%%%%%%%%%%%%%%%%%%%%%%%%%%%%%%%%%%%%%%%%%%%%%%
%%%%%%%%%%%%%%%%%%%%%%%%%%%%%%%%%%%%%%%%%%%%%%%%%%%%%%%%%%%%%%%%%%%%%%%%%%%%%%%%%%%%%%%%%%%%%%%%%%%%%%%%%%%%%%%%%%%%%%%%%%%%%%%%%%%%
\section{Preliminaries and main results}\label{sec2}

 The model of two-type CBIRE-processes can be seen as a combination of the two-type CBI-processes and the single-type CBIRE-processes. The branching and immigration mechanisms are inherited from the classical two-type CBI-processes, see pp.44-45 in \cite{li}. For ease of notation, we still use $\phi$ and $\psi$ in the two-dimensional version without confusion. Specifically, $b=(b_{ij})$ is a $(2\times2)$-matrix with
\beqlb\label{1213-1}
b_{12}+\int_{\mathbb{R}_+^2}z_2m_1(\d z)\le 0,\qquad b_{21}+\int_{\mathbb{R}_+^2}z_1m_2(\d z)\le 0,
\eeqlb
 where $m_1,m_2$ are $\sigma$-finite measures on $\mathbb{R}_+^2$ supported by $\mathbb{R}_+^2\setminus\{\mathbf{0}\}$ satisfying
\begin{eqnarray*}
 \int_{\mathbb{R}_+^2} (z_1\wedge z_1^2 + z_2)m_1(\d z) + \int_{\mathbb{R}_+^2} (z_2\wedge z_2^2 + z_1)m_2(\d z) < \infty.
 \end{eqnarray*}
The branching mechanism $\phi=(\phi_1,\phi_2)$ is a function from $\mathbb{R}_+^2$ to itself with the following representations,
\begin{equation}\label{phi 1'}
 \phi_1(\lambda) = b_{11}\lambda_1 + b_{12}\lambda_2 + c_1 \lambda_1^2 + \int_{\mathbb{R}_+^2} (\e^{-\langle\lambda,z\rangle} - 1 + \langle\lambda,z\rangle)m_1(\d z),
 \end{equation}
 \begin{equation}\label{phi 2'}
 \phi_2(\lambda) = b_{21}\lambda_1 + b_{22}\lambda_2 + c_2 \lambda_2^2 + \int_{\mathbb{R}_+^2} (\e^{-\langle\lambda,z\rangle} - 1 + \langle\lambda,z\rangle)m_2(\d z).
 \end{equation}
And we write $\psi$\ for the immigration mechanism. It is a function from\ $\mathbb{R}_+^2$\ to\ $\mathbb{R}_+$\ with representation
\begin{equation}\label{122003}
\psi(\lambda)=\langle h,\lambda\rangle + \int_{\mathbb{R}_+^2\setminus\{\mathbf{0}\}}(1-\e^{-\langle\lambda, z\rangle})n(\d z),\quad \lambda\in\mathbb{R}_+^2.
\end{equation}
In the above, $h=(h_1,h_2), c=(c_1,c_2)$ are constants in $\mathbb{R}^2_+$, and $n$ is a $\sigma$-finite measure on $\mathbb{R}_+^2$ supported by $\mathbb{R}_+^2\setminus\{\mathbf{0}\}$ satisfying$  \int_{\mathbb{R}_+^2}(1\wedge\| z\|)n(\d z )< \infty.$ The notation $\langle\cdot,\cdot\rangle$ stands for the inner product.

Let $(\Omega, \mathscr{F}, \mathscr{F}_t, \mbf{P})$ be a filtered probability space satisfying the usual hypothesis. The random environment is described by an $(\mathscr{F}_t)$-L\'evy process $\{\xi(t):0<t<\infty\}$  with $\xi(0)=0,$ whose L\'{e}vy-It\^{o} decomposition is given as follows:
\begin{equation}\label{191006xi}
  \xi(t) = at + \sigma W(t) + \int_0^t\int_{[-1,1]}z\tilde{N}(\d s,\d z) + \int_0^t\int_{[-1,1]^c}zN(\d s,\d z),\ t\ge 0,
\end{equation}
where $a \in \mathbb{R}$ and $\sigma \geq 0$ are given constants, $\{W_t: t \geq 0\}$ is an $(\mathscr{F}_t)$-Brownian motion and $N(\d s, \d z)$ is an $(\mathscr{F}_t)$-Poisson random measure on $(0, \infty)\times \mathbb{R}$ with intensity $\d s \nu(\d z)$ satisfying $\int_{(0,\infty)}(1\wedge z^2) \nu(\d z)<\infty.$ Similar to the treatment in the single-type model of CBIRE-processes \cite{helixu}, the environment can be extended to $\{\xi(t):-\infty<t<\infty\}.$

 Given an interval~$I\subset\mathbb{R},$~for~$t\in I,\lambda\in{\mathbb{R}_+^2},$~there exists~$r\mapsto u_{r,t}(\xi,\lambda)\in\mathbb{R}_+^2$~as the unique positive strong solution to
 \begin{equation}\label{191006u}
   u_{r,t}^{(i)}(\xi,\lambda) = \lambda_i - \int_r^t \e^{\xi(s)}\phi_i(e^{-\xi(s)}u_{s,t}(\xi,\lambda))\d s,\quad i=1,2,\quad r\in I\cap (-\infty,t].
 \end{equation}
Define~$(v_{r,t}(\xi,\cdot))_{t\ge r\in I}$ by
 \begin{equation}\label{191006v}
   v_{r,t}^{(i)}(\xi,\lambda_1,\lambda_2):=\e^{-\xi(r)}u_{r,t}^{(i)}(\xi,\e^{\xi(t)}\lambda_1, \e^{\xi(t)}\lambda_2),\quad \lambda=(\lambda_1,\lambda_2)\in\mathbb{R}_+^2,\quad i=1,2.
 \end{equation}
We call $u_{r,t}$ and $v_{r,t}$ random cumulant semigroups. Define a stochastic transition semigroup $Q_{r,t}^{\xi}(x,\d y)$ by
\begin{equation}\label{191007q}
   \int_{\mathbb{R}_+^2}\e^{-\langle\lambda,y\rangle}Q_{r,t}^{\xi}(x,\d y) = \exp\{-\langle x,v_{r,t}(\xi,\lambda)\rangle\}, \qquad \lambda,x \in \mathbb{R}_+^2.
 \end{equation}
It is easy to check that $Q_{r,t}^{\xi}(x,\cdot)$ has the branching property, namely,
\begin{equation}\label{122001}
Q_{r,t}^{\xi}(x+y,\cdot)=Q_{r,t}^{\xi}(x,\cdot)\ast Q_{r,t}^{\xi}(y,\cdot), r<t\in\mathbb{R}, x,y\in\mathbb{R}_+^2,
\end{equation}
where $\ast$ means the convolution between measures. And we write $Q_{t}^{\xi}(x,\cdot)$ for the special case of $r=0.$ Furthermore, for $t>0,$ define the transition semigroup $\bar{Q}_t(x,\cdot)$ by
\begin{equation}\label{qbar}
\int_{\mathbb{R}_+^2}e^{-\langle\lambda,y\rangle}\bar{Q}_t(x,\d y)=\mathbf{P}\Big[\exp\Big\{-\langle x,v_{0,t}(\xi,\lambda)\rangle-\int_0^t\psi(v_{s,t}(\xi,\lambda))\d s\Big\}\Big].
\end{equation}

In fact, $\bar{Q}_t(x,\cdot)$ is the transition semigroup of $\{X(t):t\ge 0\},$ which is a two-type CBIRE-process with branching mechanism $\phi,$ immigration semigroup $\psi,$ random environment $\xi$ and initial value $x\in\mathbb{R}_+^2.$

When $\eta(t)\equiv (0,0),$ we get a two-type CBRE-process $\{Y(t):t\ge 0\},$ and its transition semigroup ${Q}_t(x,\cdot)$ is given by
\begin{equation}\label{q}
\int_{\mathbb{R}_+^2}e^{-\langle\lambda,y\rangle}{Q}_t(x,\d y)=\mathbf{P}\big[\exp\big\{-\langle x,v_{0,t}(\xi,\lambda)\rangle\big\}\big].
\end{equation}

 According to Theorem 2.2 in \cite{qinzheng}, if $\mathbf{P}[\xi(1)]<0,$ $\int_{\| z\|\ge 1}\log(\| z\|)n(\d z)<\infty$ and the eigenvalues of~$b$~have strictly positive real parts, there is a unique limiting distribution $\mu$ for $X(t)$ as $t\rightarrow\infty.$ And for $Y(t)$ the limiting distribution $\mu=\delta_{\mathbf{0}},$ the Dirac measure for $(0,0).$

 The main purpose of this paper is to prove that the convergence still holds in {\it total variation distance}\ and is exponential in {\it $L^1$-Wasserstein distance}. By ${\mathcal{P}(\mbb{R}_{+}^2)}$ we denote the space of all Borel probability measures over $\mbb{R}_{+}^2$. Let $d$ be a metric on $\mbb{R}_{+}^2$ such that $(\mbb{R}_{+}^2,d)$ is a Polish space and define
\beqlb
{\mathcal{P}}_{d}(\mbb{R}_{+}^2)=\Big\{\rho\in{\mathcal{P}}(\mbb{R}_{+}^2):
\int_{\mbb{R}_{+}^2}d(x,0)\,\rho(\mrm{d}x)<\infty\Big\}.
\eeqlb
The Wasserstein distance on ${\mathcal{P}_{d}(\mbb{R}_{+}^2)}$ is defined by
\beqlb
W_{d}(P_1,P_2)=\inf_{\Pi\in\mathcal{C}(P_1,P_2)}\int_{\mbb{R}_{+}^2\times \mbb{R}_{+}^2}d(x,y)\,\Pi(\mathrm{d}x,\mathrm{d}y),
\eeqlb
where $\mathcal{C}(P_1,P_2)$ stands for the set of all coupling measures of $P_1$ and $P_2$, i.e. $\mathcal{C}(P_1,P_2)$ is the collection of measures on $\mbb{R}_{+}^2\times\mbb{R}_{+}^2$ having $P_1$ and $P_2$ as marginals. It can be shown that this infimum is attainable. According to Theorem 6.16 in \cite{villani}, $({\mathcal{P}_{d}(\mbb{R}_{+}^2)},W_{d})$ is also a Polish space. In the reminder of the article, we will use the following particular examples.

If $d_{TV}(x,y)=\mbf{1}_{\{x\neq y\}},$ then ${\mathcal{P}}_{d_{TV}}(\mbb{R}_{+}^2)={\mathcal{P}}(\mbb{R}_{+}^2)$ and
$$
W_{d_{TV}}(P_1,P_2)=\frac{1}{2}\| P_1-P_2\|_{TV}:=\frac{1}{2}\sup_{A\subset{\mathcal{B}}(\mbb{R}_{+}^2)}
\vert P_1(A)-P_2(A)\vert
$$
is the {\it total variation distance}.

If $d_1(x,y)=\vert x-y\vert$, then
$$
{\mathcal{P}}_{d_1}(\mbb{R}_{+}^2)=\Big\{\rho\in{\mathcal{P}}(\mbb{R}_{+}^2):
\int_{\mbb{R}_{+}^2}\vert x\vert\,\rho(\mrm{d}x)<\infty\Big\}
$$
and the corresponding $W_{d_1}$ is the\ {\it $L^1$-Wasserstein distance} written as $W_{d_1}:=W_1$.

For a matrix $A,$ we denote its determinant by $det(A),$ and its trace by $tr(A).$ Moreover, define $\beta := a + \frac{\sigma^2}{2} + \int_{[-1,1]}(\e^{z} - 1 - z)\nu(\d z) + \int_{[-1, 1]^c}(\e^{z} - 1)\nu(\d z),{ \Delta}:=[tr({b})]^2-4\ det ({b})>0$ and $\epsilon:=\sqrt{\Delta}+b_{22}-b_{11}+2b_{21}$.

Let $d^{\theta}$ be a metric on $\mbb{R}_{+}^2$ given by
$$
d^{\theta}(x,y)=(1+\theta)\vert x-y\vert
$$
for some positive constant $\theta$. Obviously, $(\mbb{R}_{+}^2,d^{\theta})$ is a complete separable metric space and  $({\mathcal{P}_{d^{\theta}}(\mbb{R}_{+}^2)},W_{d^{\theta}})$ is also a Polish space.

\begin{thm}
Suppose that $\int_{\| z\|\ge 1}\| z\|{ n(\d z)}<\infty,$ $\int_1^{\infty}\e^z{\nu(\d z)}<\infty,$ and ${ \beta}<\frac{1}{2}(tr({b})-\sqrt{{ \Delta}}).$ Then there exist positive constants $\rho>0,\theta>0$ and a unique stationary distribution $\mu\in\mathcal{P}_{1}(\mbb{R}_{+}^2)$ such that for any $x\in\mbb{R}_{+}^2$ and $t\geq0,$
$$
W_1( \delta_x\bar Q_t,\mu)\leq W_{d^{\theta}}(\delta_x,\mu)\e^{-\rho t},
$$
where $\delta_x\bar{Q}_t=\bar{Q}_t(x,\cdot)$. Moreover, the Laplace transform of $\mu$ is
$$
\int_{\mathbb{R}_+^2}\e^{-\langle\lambda,y\rangle}\mu(\d y)=\mathbf{P}\big[\exp\big\{-\int_{-\infty}^0\psi(v_{s,0}(\xi,\lambda))\d s\big\}\big].
$$
\end{thm}

\begin{thm}
Suppose that Condition \ref{cond} (see section 4 for details) is satisfied. If $\int_{\| z\|\ge 1}\| z\|{ n(\d z)}<\infty,$ $\int_1^{\infty}\e^z{\nu(\d z)}<\infty,$ ${ \beta}<\frac{1}{2}(tr({b})-\sqrt{{ \Delta}})$ and $\liminf\limits_{t\rightarrow\infty}\xi(t)=-\infty,$ then there exists a unique stationary distribution $\mu$ such that for any $x\in\mbb{R}_{+}^2$ and $t\geq0,$
$$\lim\limits_{t\rightarrow\infty}\|\delta_x\bar Q_t-\mu\|_{TV}=0.$$
\end{thm}
%%%%%%%%%%%%%%%%%%%%%%%%%%%%%%%%%%%%%%%%%%%%%%%%%%%%%%%%%%%%%%%%%%%%%%%%%%%%%%%%%%%%%%%%%%%%%%%%%%%%%%%%%%%%%%%%%%%%%%%%%%%%%%%%%%%%
%%%%%%%%%%%%%%%%%%%%%%%%%%%%%%%%%%%%%%%%%%%%%%%%%%%%%%%%%%%%%%%%%%%%%%%%%%%%%%%%%%%%%%%%%%%%%%%%%%%%%%%%%%%%%%%%%%%%%%%%%%%%%%%%%%%%
\section{Proofs }\label{sec3}

\subsection{Exponential ergodicity  in the $ L_1$-Wasserstein distance }\label{subsec2}

 In this section we give the proofs of the exponential ergodicity in  $ L_1$-Wasserstein distance. Our method of coupling is strongly influenced by \cite{chenli,li2}, and the key to the proof is the use of conditional branching property. For $0<r<t,$ define
\begin{eqnarray*}
\pi'_1(r,t)\ar=\ar \frac{\partial u_{r,t}^{(1)}(\xi,\lambda)}{\partial \lambda_1}\vert_{\lambda=\mathbf{0+}}+\frac{\partial u_{r,t}^{(1)}(\xi,\lambda)}{\partial \lambda_2}\vert_{\lambda=\mathbf{0+}},\cr
\pi'_2(r,t)\ar=\ar\frac{\partial u_{r,t}^{(2)}(\xi,\lambda)}{\partial \lambda_1}\vert_{\lambda=\mathbf{0+}}+\frac{\partial u_{r,t}^{(2)}(\xi,\lambda)}{\partial \lambda_2}\vert_{\lambda=\mathbf{0+}}.
\end{eqnarray*}
Differentiating both sides of \eqref{191006u}, we get,
 \begin{eqnarray*}\label{1213-3}
   \frac{\partial u_{r,t}^{(1)}(\xi,\lambda)}{\partial \lambda_1}\vert_{\lambda=\mathbf{0+}} &=& 1-\int_r^t\Big(b_{11}\frac{\partial u_{s,t}^{(1)}(\xi,\lambda)}{\partial \lambda_1}\vert_{\lambda=\mathbf{0+}}+b_{12}\frac{\partial u_{s,t}^{(2)}(\xi,\lambda)}{\partial \lambda_1}\vert_{\lambda=\mathbf{0+}}\Big)\d s, \nnm\\
   \frac{\partial u_{r,t}^{(2)}(\xi,\lambda)}{\partial \lambda_1}\vert_{\lambda=\mathbf{0+}} &=& -\int_r^t\Big(b_{21}\frac{\partial u_{s,t}^{(1)}(\xi,\lambda)}{\partial \lambda_1}\vert_{\lambda=\mathbf{0+}}+b_{22}\frac{\partial u_{s,t}^{(2)}(\xi,\lambda)}{\partial \lambda_1}\vert_{\lambda=\mathbf{0+}}\Big)\d s, \nnm\\
     \frac{\partial u_{r,t}^{(2)}(\xi,\lambda)}{\partial \lambda_2}\vert_{\lambda=\mathbf{0+}} &=& 1-\int_r^t\Big(b_{22}\frac{\partial u_{s,t}^{(2)}(\xi,\lambda)}{\partial \lambda_2}\vert_{\lambda=\mathbf{0+}}+b_{21}\frac{\partial u_{s,t}^{(1)}(\xi,\lambda)}{\partial \lambda_2}\vert_{\lambda=\mathbf{0+}}\Big)\d s, \nnm\\
       \frac{\partial u_{r,t}^{(1)}(\xi,\lambda)}{\partial \lambda_2}\vert_{\lambda=\mathbf{0+}} &=& -\int_r^t\Big(b_{12}\frac{\partial u_{s,t}^{(2)}(\xi,\lambda)}{\partial \lambda_2}\vert_{\lambda=\mathbf{0+}}+b_{11}\frac{\partial u_{s,t}^{(1)}(\xi,\lambda)}{\partial \lambda_2}\vert_{\lambda=\mathbf{0+}}\Big)\d s.
 \end{eqnarray*}
Then, it is not difficult to see,
\begin{eqnarray*}
\pi'_1(r,t)\ar=\ar 1-\int_r^t[b_{11}\pi'_1(s,t)+b_{12}\pi'_2(s,t)]\d s,\cr
\pi'_2(r,t)\ar=\ar 1-\int_r^t[b_{22}\pi'_2(s,t)+b_{21}\pi'_1(s,t)]\d s.
\end{eqnarray*}
For $0<r<t,$ set
\beqlb\label{1.5}
\pi_1(r,t)\ar=\ar\mathbf{P}\frac{\partial v_{r,t}^{(1)}(\xi,\lambda)}{\partial \lambda_1}\vert_{\lambda=\mathbf{0+}}+\mathbf{P}\frac{\partial v_{r,t}^{(1)}(\xi,\lambda)}{\partial \lambda_2}\vert_{\lambda=\mathbf{0+}},\cr
\pi_2(r,t)\ar=\ar\mathbf{P}\frac{\partial v_{r,t}^{(2)}(\xi,\lambda)}{\partial \lambda_1}\vert_{\lambda=\mathbf{0+}}+\mathbf{P}\frac{\partial v_{r,t}^{(2)}(\xi,\lambda)}{\partial \lambda_2}\vert_{\lambda=\mathbf{0+}}.
\eeqlb

\begin{lem} Suppose that $\int_1^{\infty}\e^{z}\nu(\d z)<\infty.$ For $t\ge0$, \begin{equation}\label{1214-1}
\pi(0,t)=\e^{\beta t}\pi'(0,t).
\end{equation}
\end{lem}
\begin{proof}
By Lemma 3.2 in \cite{jizheng}, if $\int_1^{\infty}\e^{z}\nu(\d z)<\infty,$ then $\mathbf{P} \e^{\xi(t)}=\e^{\beta t}$ for all $t\ge 0$. And for all~$t\ge r\in I,$
$$\frac{\partial v_{r,t}(\xi,\lambda)}{\partial \lambda}\vert_{\lambda=\mathbf{0}+} =\e^{\xi(t)-\xi(r)} \frac{\partial u_{r,t}(\xi,\lambda)}{\partial \lambda}\vert_{\lambda=\mathbf{0}+}.$$ 
Since $\frac{\partial u_{r,t}(\xi,\lambda)}{\partial \lambda}\vert_{\lambda=\mathbf{0}+}=\e^{{b}(r-t)}$ is not random, where
$$\e^{b(r-t)}:=I_{2\times 2}+b(r-t)+\frac{(r-t)^2}{2!}b^2+\cdots+\frac{(r-t)^k}{k!}b^k+\cdots,$$
 and $I_{2\times 2}$ is the $2\times 2$ identity matrix. Thus (\ref{1214-1}) holds.
\end{proof}
\begin{prop}\label{prop3.2}
Suppose that $\int_{\|z\|\ge 1}\|z\|n(\d z)<\infty.$ Let $(\bar Q_t)_{t\geq0}$ be the transition semigroup of a two-type CBIRE-process. Then for all $x,y\in\mbb{R}_{+}^2$ and $t\geq0$ we have
\beqlb\label{1214-2}
\vert\langle x-y,\pi(0,t)\rangle\vert\leq W_1(\delta_x\bar Q_t,\delta_y
\bar Q_t)\leq\sum_{i=1}^2\vert x_i-y_i\vert\pi_i(0,t),
\eeqlb
\end{prop}

\begin{proof} %The proof is based on the same idea as that of Theorem 2.2 in Li (2020+).
  For $i=1,2,$ denote $\gamma_i=h_i+\int_{\mathbb{R}_+^2}z_in(\d z).$ Taking derivatives with respect to $\lambda=\mathbf{0}+$\ on both sides of \eqref{qbar}, we get,
$$ \int_{\mathbb{R}_+^2}(y_1+y_2)\bar{Q}_t(x,\d y)=\langle x,\pi(0,t)\rangle +\int_0^t\langle \gamma,\pi(s,t)\rangle\d s.$$
It follows from Theorem 5.10 in \cite{chen} that,
$$
W_1(\delta_x\bar Q_t,\delta_y\bar Q_t)\geq\int_{\mbb{R}_{+}^2}(z_1+z_2)\,\Big(\bar Q_t(x,\mrm{d}z)-\bar Q_t(y,\mrm{d}z)\Big)=
\langle x-y,\pi(0,t)\rangle.
$$
Symmetrically, $W_1(\delta_x\bar Q_t,\delta_y\bar Q_t)\geq\langle y-x,\pi(0,t)\rangle,$ then the first inequality follows. On the other hand, for $x,y\in\mbb{R}_{+}^2,$ put $(x-y)_{\pm}:=((x_1-y_1)_{\pm},(x_2-y_2)_{\pm})$, and $x\wedge y:=x-(x-y)_{+}=y-(x-y)_{-}.$
Let $Q^{\xi}_t(x,y,\mrm{d}\eta_1,\mrm{d}\eta_2)$ be the image of the product measure
$$
Q_t^{\xi}(x\wedge y,\mrm{d}\gamma_0)Q_t^{\xi}((x-y)_{+},\mrm{d}\gamma_1)Q_t^{\xi}((x-y)_{-},\mrm{d}\gamma_2)
$$
under the mapping $(\gamma_0,\gamma_1,\gamma_2)\mapsto(\eta_1,\eta_2):=(\gamma_0+\gamma_1,\gamma_0+\gamma_2)$. Define $Q_t(x,y,\mrm{d}\eta_1,\mrm{d}\eta_2)$ on $\mbb{R}_{+}^4$ by
\begin{equation}\label{branch}
Q_t(x,y,\mrm{d}\eta_1,\mrm{d}\eta_2)=\mbf{P}Q^{\xi}_t(x,y,\mrm{d}\eta_1,\mrm{d}\eta_2).
\end{equation}
It's not hard to see that
$Q_t(x,y,\mrm{d}\eta_1,\mrm{d}\eta_2)$ is a coupling of $Q_t(x,\mrm{d}\eta_1)$ and $Q_t(y,\mrm{d}\eta_2)$.

{ By \eqref{122001} and Theorem 1.35 in \cite{li}, for $x\in\mathbb{R}_+^2$ and $r\in[0,t]$ there exists $a^x_{r,t}\in\mathbb{R}_+^2$ and a finite measure $(1\wedge \vert z\vert)l^x_{r,t}(\d z),$ such that,
\begin{equation}\label{122002} \langle x,v_{r,t}(\zeta,\lambda)\rangle=\langle a_{r,t}^x,\lambda\rangle+\int_{\mathbb{R}_+^2\setminus\{\mathbf{0}\}}(1-\e^{-\langle\lambda, z\rangle})l_{r,t}^x(\d z),\quad \lambda\in\mathbb{R}_+^2,
\end{equation}
where $v_{r,t}(\zeta,\lambda)$ is defined by \eqref{191007q} with $\xi$ replaced by a  c\`{a}dl\`{a}g function $\zeta=\{\zeta(t):t\in \mathbb{R}\}.$
By \eqref{122003}, \eqref{122002} and  Theorem 1.37 in \cite{li}, for $r\in[0,t]$, there exists $A_{r,t}\in\mathbb{R}_+^2$ and a finite measure $(1\wedge \vert z\vert)L_{r,t}(\d z),$ such that,
\begin{equation}\label{122004}
J_{r,t}(\zeta,\lambda):=\int_r^t\psi(v_{s,t}(\zeta,\lambda))\d s
=\langle A_{r,t},\lambda\rangle+\int_{\mathbb{R}_+^2\setminus\{\mathbf{0}\}}(1-\e^{-\langle\lambda, z\rangle})L_{r,t}(\d z),\quad \lambda\in\mathbb{R}_+^2.
\end{equation}
For $a_s=(a^{(1)}_s,a^{(2)}_s)\in\mathbb{R}^2$,  $\int_r^t a_s\d s=\big(\int_r^t a^{(1)}_s\d s,\int_r^t a^{(2)}_s \d s\big).$ In fact, by taking $x=h$ in \eqref{122002}, we can calculate that,
$$A_{r,t}=\int_r^t a_{s,t}^h \d s,$$
$$L_{r,t}(\d z)=\int_r^t\big[l_{s,t}^h(\d z)+\int_{\mathbb{R}_+^2\setminus\{\mathbf{0}\}}Q^{\zeta}_{s,t}(y,\d z)n(\d y)\big]\d s.$$
Following Theorem 1.35 in \cite{li}, for a c\`{a}dl\`{a}g function $\zeta$ %=\{\zeta(t):t\in I\subseteq\mathbb{R}\}$
and $r\in[0,t]$, we can define an infinitely divisible measure $\Upsilon^{\zeta}_{r,t}$ on $\mathbb{R}_+^2$ by \begin{equation}\label{122005}
\int_{\mathbb{R}_+^2}\e^{-\langle \lambda,y\rangle}\Upsilon^{\zeta}_{r,t}(\d y)=J_{r,t}(\zeta,\lambda).
\end{equation}
It is easy to verify that
$$J_{r,t}(\zeta,\lambda)=J_{r,s}(\zeta,v_{s,t}(\zeta,\lambda))+J_{s,t}(\zeta,\lambda).$$
When the function $\zeta$ reduces into the case $\zeta(t)\equiv 0,$ $\big(v_{r,t}(\zeta,\cdot)\big)_{r\le t}$ goes back to the cumulant semigroup of a classical CB-process, and $(\Upsilon^{\zeta}_{r,t})_{r\le t }$ goes back to the skew convolution semigroup of a classical CBI-process.

 Let $(\Upsilon^{\zeta}_{r,t})_{r\le t }$ be the \emph{random skew convolution semigroup} associated with $(Q^{\xi}_{r,t})_{r\le t }$ defined by \eqref{122005} with $\zeta=\xi.$} Let $\bar Q_{r,t}^{\xi}(x,\cdot):= Q_{r,t}^{\xi}(x,\cdot)\ast\Upsilon^{\xi}_{r,t},\ r\le t.$
It is clear that for $x\in\mathbb{R}_+^2,t\ge 0,$
 \begin{equation}\label{1214-3}
 \bar Q_t(x,\cdot)=\p [\bar Q_{0,t}^{\xi}(x,\cdot)].
 \end{equation}
 For more details on skew convolution semigroups, see Chapter 9 in \cite{li}. Let $\bar Q^{\xi}_t(x,y,\mrm{d}\sigma_1,\mrm{d}\sigma_2)$ be the image of $\Upsilon_t(\d\eta_0)Q^{\xi}_t(x,y,\mrm{d}\eta_1,\mrm{d}\eta_2)$ under the mapping $(\eta_0,\eta_1,\eta_2)\mapsto(\sigma_1,\sigma_2)=(\eta_0+\eta_1,\eta_0+\eta_2).$ Define $\bar Q_t(x,y,\mrm{d}\sigma_1,\mrm{d}\sigma_2)$ on $\mbb{R}_{+}^4$ by
\begin{equation}\label{branch1}
\bar Q_t(x,y,\mrm{d}\sigma_1,\mrm{d}\sigma_2):=\mbf{P}\bar Q^{\xi}_t(x,y,\mrm{d}\sigma_1,\mrm{d}\sigma_2).
\end{equation} Then by the relation \eqref{1214-3}, we can verify that $\bar Q_t(x,y,\mrm{d}\sigma_1,\mrm{d}\sigma_2)$ is the coupling measure of $\bar Q_t(x,\d  \sigma_1)$ and $\bar Q_t(y,\d  \sigma_2).$ For similar construction in the setting of measure-valued processes, see \cite{li2}.
Finally,
\begin{eqnarray}\label{12801}
&&W_1(\delta_x\bar Q_t,\delta_y\bar Q_t)\leq\int_{\mbb{R}_{+}^2\times\mbb{R}_{+}^2}\vert\sigma_1-\sigma_2\vert\bar Q_t(x,y,\mrm{d}\sigma_1,\mrm{d}\sigma_2)\cr
&&\quad\leq\mbf{P}\Big(\int_{\mbb{R}_{+}^2}(\gamma_{11}+\gamma_{12}) Q^{\xi}_t((x-y)_{+},\mrm{d}\gamma_1)
+\int_{\mbb{R}_{+}^2}(\gamma_{21}+\gamma_{22}) Q^{\xi}_t((x-y)_{-},\mrm{d}\gamma_2)\Big)\cr
&&\quad=\mbf{P}\Big(\int_{\mbb{R}_{+}^2}(\zeta_1+\zeta_2)\, Q^{\xi}_t((\vert x_1-y_1\vert,\vert x_2-y_2\vert), \mrm{d}\zeta)\Big)\cr
&&\quad=\int_{\mbb{R}_{+}^2}(\zeta_{1}+\zeta_{2})\, Q_t((\vert x_1-y_1\vert,\vert x_2-y_2\vert),\mrm{d}\zeta) \cr &&\quad=\sum_{i=1}^2\vert x_i-y_i\vert\pi_i(0,t),
\end{eqnarray}
where the first equality comes from the conditioned branching property.
\end{proof}

We are now in a position to prove our first main result.

\noindent\textbf{Proof of Theorem 2.1} Because $\p[\xi(1)]<{ \beta}<\frac{1}{2}(tr({b})-\sqrt{{ \Delta}}),$ and $\frac{1}{2}(tr({b})-\sqrt{{ \Delta}})$ is the smallest eigenvalue of $b,$ we may adjust the parameters so that $\p[\xi(1)]<0,$ and the eigenvalues of $b$ are strictly positive. According to Equation (10-11) in \cite{qinzheng}, this adjustment is without loss of generality. Thus, by the result of Theorem 2.2 in \cite{qinzheng}, for~$x\in\mathbb{R}_+^2,$ ${Q}_t(x,\cdot)$~converges weakly to~$\mu$~as~$t\rightarrow\infty.$

 By assumption, the equation
$$
\lambda^2+(b_{11}+b_{22})\lambda+b_{11}b_{22}-b_{12}b_{21}=0
$$
has two different solutions on $\mbb{R}_{-}$, which we denote by $\lambda_1$ and $\lambda_2$. Recall that $\Delta=(b_{22}-b_{11})^2+4b_{12}b_{21}>0$, we have
$\lambda_1=\frac{1}{2}(-b_{11}-b_{22}+\sqrt{\Delta}),\lambda_2=\lambda_1-\sqrt{\Delta}.$ Note that \eqref{1213-1} implies $b_{12},b_{21}\le 0$.

{\bf Case 1:} $b_{12}b_{21}=0$. it is straightforward to show that
$$
\pi'_1(0,t)=\mrm{e}^{-b_{11}t},\quad\pi'_2(0,t)=\mrm{e}^{-b_{22}t}.
$$
%and then it is easy to see that \eqref{1213-2} holds with $\delta=2$ and $\lambda=\min\{b_{11},b_{22}\}$.
%In the following we only focus on the case with

{\bf Case 2:} $b_{12}b_{21}\neq0$. Some simple calculations yield
\beqnn
&&\pi'_1(0,t)=\frac{b_{12}(b_{22}-b_{11}+2b_{21}-\sqrt{\Delta})}{\sqrt{\Delta}(\sqrt{\Delta}+b_{11}-b_{22})}
\mrm{e}^{\lambda_1t}-\frac{b_{12}(\sqrt{\Delta}+b_{22}-b_{11}+2b_{21})}
{\sqrt{\Delta}(-\sqrt{\Delta}+b_{11}-b_{22})}
\mrm{e}^{\lambda_2t},\\
&&\pi'_2(0,t)=\frac{b_{11}-b_{22}-2b_{21}+\sqrt{\Delta}}{2\sqrt{\Delta}}\mrm{e}^{\lambda_1t}
+\frac{\sqrt{\Delta}+b_{22}-b_{11}+2b_{21}}{2\sqrt{\Delta}}
\mrm{e}^{\lambda_2t}.
\eeqnn
If $\epsilon:=\sqrt{\Delta}+b_{22}-b_{11}+2b_{21}<0$, then there exist $\theta_{11}=-\frac{b_{12}(2\sqrt{\Delta}-\epsilon)}
{\sqrt{\Delta}(\sqrt{\Delta}+b_{11}-b_{22})}\in(0,1)$ and $\theta_{12}=\frac{2\sqrt{\Delta}-\epsilon}{2\sqrt{\Delta}}>0$ such that
$$
\pi'_1(0,t)=\theta_{11}\mrm{e}^{\lambda_1t}+(1-\theta_{11})\mrm{e}^{\lambda_2t},\,
\pi'_2(0,t)\leq\theta_{12}\mrm{e}^{\lambda_1t}+\mrm{e}^{\lambda_2t}.
$$
If $\epsilon>0$, there exist  $\theta_{21}=-\frac{b_{12}(2\sqrt{\Delta}-\epsilon)}{\sqrt{\Delta}(\sqrt{\Delta}+b_{11}-b_{22})}>0$ and $\theta_{22}=\frac{2\sqrt{\Delta}-\epsilon}{2\sqrt{\Delta}}\in(0,1)$ such that
$$
\pi'_1(0,t)\leq\theta_{21}\mrm{e}^{\lambda_1t}+\mrm{e}^{\lambda_2t},\,
\pi'_2(0,t)=\theta_{22}\mrm{e}^{\lambda_1t}+(1-\theta_{22})\mrm{e}^{\lambda_2t}.
$$
If $\epsilon=0$, then $\pi'_1(0,t)=\pi'_2(0,t)=\mrm{e}^{\lambda_1t}.$

Then, for all $t\geq0$ and $x,y\in\mbb{R}_{+}^2$, there exists $\theta=\max\{\theta_{12},\theta_{21}\}$ such that
\beqlb\label{1213-2}
\sum_{i=1}^2\vert x_i-y_i\vert\pi_i'(0,t)\leq d^{\theta}(x,y)\mrm{e}^{\lambda_1t}
\eeqlb
Notice that under the assumption $\int_{\| z\|\ge 1}\| z\|{ n(\d z)}<\infty$ and ${\beta}<\frac{1}{2}(tr({b})-\sqrt{{ \Delta}}),$ the limiting distribution $\mu$ has finite expectation. Therefore,
$W_1( \delta_x\bar Q_t,\mu),W_{d^{\theta}}(\delta_x,\mu)$ are well-defined.
By equation \eqref{1214-1}, \eqref{1214-2}, \eqref{1213-2} and the convexity of the Wasserstein distance, we obtain,
{\beqnn
&W_1(\delta_x\bar Q_t,\mu)&\leq\int_{\mbb{R}_{+}^4} W_1(\delta_x\bar Q_t,\delta_y\bar Q_t)\,H(\mrm{d}x,\mrm{d}y)\\
&&\leq\mrm{e}^{(\beta+\lambda_1)t}\int_{\mbb{R}_{+}^4}W_{d^{\theta}}(x,y)\,H(\mrm{d}x,\mrm{d}y)\\
&&=\mrm{e}^{(\beta+\lambda_1)t}W_{d^{\theta}}(\delta_x,\mu),
\eeqnn}
where $H$ is the optimal coupling measure of $(\delta_x,\mu)$; see for instance, Chapter 5 in \cite{chen}. Observing that $\lambda_1=-\frac{1}{2}(tr({b})-\sqrt{{ \Delta}}),$ we arrive at the conclusion by setting $\rho=\frac{1}{2}(tr({b})-\sqrt{{\Delta}})-\beta>0.$\qed

{\begin{cor}
Suppose that $\int_1^{\infty}\e^z{\nu(\d z)}<\infty,$ ${ \beta}<\frac{1}{2}(tr({b})-\sqrt{{ \Delta}}).$ Then there exist $\rho>0,\theta>0$ such that for any $x\in\mbb{R}_{+}^2$ and $t\geq0,$
$$
W_1( \delta_xQ_t,\delta_{0})\leq W_{d^{\theta}}(\delta_x,\delta_{0})\e^{-\rho t}.
$$
\end{cor}}
%%%%%%%%%%%%%%%%%%%%%%%%%%%%%%%%%%%%%%%%%%%%%%%%%%%%%%%%%%%%%%%%%%%%%%%%%%%%%%%%%%%%%%%%%%%%%%%%%%%%%%%%%%%%%%%%%%%%%%%%%%%%%%%%%%%%
\subsection{ Ergodicity in the total variance distance}\label{subsec22}

In this section, we prove the ergodicity in the total variance distance. The key to the proof is the finiteness of the random cumulant semigroup. And we deal with it using the tool of Dawson-Watanabe superprocess. For the sake of completeness, we make a brief introduction to the inhomogeneous superprocesses, which is mainly summerized from \cite{li}.

 Suppose that $\tilde E$ is a Borel subset of $I \times F,$ where $I \subset \mathbb{R}_{+}$ is an interval and $F$ is a Lusin topological space.  $(P_{r, t}: t \geq r \in I)$ is an inhomogeneous Borel right transition semigroup with global state space $\tilde{E}$. The system $\Pi=(\Omega, \mathscr{F}, \mathscr{F}_{r, t}, \Pi_{t}, \mathbf{P}_{r, x})$ is a right continuous inhomogeneous Markov process realizing $(P_{r, t}: t \geq r \in I)$. For any $t \in I$ let $I_{t}=[0, t] \cap I$ and $E_{t}=\{x \in F:(t, x) \in$ $\tilde{E}\}$. According to Theorem 6.10 in \cite{li}, for every $t \in I$ and $f \in B(E_{t})^{+},$ there is a unique bounded positive solution $(r, x) \mapsto v_{r, t}(x)=V_{r, t} f(x)$ to the integral equation
\begin{equation}\label{liinhomo}
v_{r, t}(x)=\mathbf{P}_{r, x}[f(\Pi_{t})]-\int_{r}^{t} \mathbf{P}_{r, x}[\Phi(s, \Pi_{s}, v_{s, t})] \d s, \quad r \in I_{t}, x \in E_{r}.
\end{equation}
In the equation, $\Phi$ is an inhomogeneous branching mechanism defined as
\begin{eqnarray}\label{eqPhi}
\begin{aligned}
\Phi(s, x, f)=& b(s, x) f(x)+c(s, x) f(x)^{2}-\int_{E_{s}} f(y) g(s, x, \d y) \\
&+\int_{M(E_{s})^{\circ}}[\mathrm{e}^{-\nu(f)}-1+\nu( f)] H(s, x, \d \nu), (s, x) \in \tilde{E},f \in B(E_{s})^{+},
\end{aligned}
\end{eqnarray}
where $b \in B(E)$ and $c \in B(E)^{+},$ $g(s, x, \d y)$ is a bounded kernel on $\tilde E.$ And $H(s, x, \d \nu)$ is a $\sigma$-finite kernel from $\tilde{E}$ to $M(\tilde{E})^{\circ} ,$ where $M(\tilde E)$ denotes the space of finite Borel measures on $\tilde E.$ For every $(s, x) \in \tilde{E}$ we assume $g(s, x, \d y)$ is supported by $\{s\} \times E_{s}$ and $H(s, x, \d \nu)$ is supposed by $M(\{s\} \times E_{s})^{\circ}$. Then $g(s, x, \d y)$ and $H(s, x, \d \nu)$ can be seen as measures on $E_{s}$ and $M(E_{s})^{\circ}, $ respectively. In addition, we assume
\begin{eqnarray*}
&\sup _{(s, x) \in \tilde{E}}\bigg[\vert b(s, x)\vert+c(s, x)+g(s, x, E_{s}) \\
&\qquad\qquad\qquad\qquad+\int_{M(E_{s})^{\circ}}[\nu(1) \wedge \nu(1)^{2}+\nu(1)] H(s, x, \d \nu)\bigg]<\infty,
\end{eqnarray*}
where $\nu_{x}(\d y)$ denotes the restriction of $\nu(\d y)$ to $E_{s} \backslash\{x\}. $
 A Dawson-Watanabe superprocess with spatial motion $\Pi$ and branching mechanism $\Phi$ is a Markov process in $M(E)$ with transition semigroup $(Q_{r,t})_{t \geq r\geq 0},$ defined by
$$
\int_{M(E_{t})} \mathrm{e}^{-\nu(f)} Q_{r, t}(\mu, \d \nu)=\exp \{-\mu(V_{r, t} f)\}, \quad f \in B(E_{t})^{+} .
$$
  In this case, we can rewrite \eqref{liinhomo} into
\begin{equation}\label{eqp}
v_{r,t}(x)=P_{r,t} f(x)-\int_{0}^{t} \d s \int_{E} \Phi(s,y, v_{s,t}) P_{s,t}(x, \d y), \quad x \in E, t\geq r \geq 0.
\end{equation}
Following by similar arguments to the proof of Corollary 5.18 in \cite{li} in inhomogeneous setting, we get the comparison theorem.
 \begin{prop}\label{com}
 Suppose that $\Phi_{1}$ and $\Phi_{2}$ are two branching mechanisms given by \eqref{eqPhi} satisfying $\Phi_{1}(x, f) \geq \Phi_{2}(x, f)$ for all $x \in E$ and $f \in B(E)^{+} .$ Let $(t, x) \mapsto v_{i}(t, x)$ be the solution of \eqref{eqp} with $\Phi$ replaced by $\Phi_{i} .$ Then $v_{1}(t, x) \leq$ $v_{2}(t, x)$ for all $t \geq 0$ and $x \in E$.
\end{prop}

%Recall that $\vert\cdot\vert$ denotes the supremum norm of functions on $E$.

In our model of two-type continuous-state branching processes in L\'{e}vy random environments, $E=\{1,2\},$ and the parameters in \eqref{eqPhi} takes the particular form
\begin{eqnarray}\label{g}
b(1)\ar=\ar b_{11}, g(s,1,f)=\int_{\{1,2\}}f(y)g(s,1,\d y)=-b_{12}f(2),\cr
b(2)\ar=\ar b_{22}, g(s,2,f)=\int_{\{1,2\}}f(y)g(s,2,\d y)=-b_{21}f(1).
\end{eqnarray}
Since $g$ is homogeneous here, we write $g(s,i,f)=g(i,f)$ for simplicity.
Furthermore, take $\Phi(x,s,f)=\e^{\xi(s)}\phi_x(e^{-\xi(s)},f),$ and $\Pi(t)\equiv\Pi(0),t\in I,$ we get \eqref{191006u}. Note that the domain of $f$ is exactly $\{1,2\},$ so we sometimes use $\lambda\in\mathbb{R}_+^2$ instead of $f$ by letting $\lambda_i=f(i),i=1,2.$

Define another branching mechanism $\tilde\phi=(\tilde\phi_1,\tilde\phi_2)$ as a function from $\mathbb{R}_+^2$ to itself with the following representations,
\begin{equation}\label{tilphi 1'}
 \tilde\phi_1(\lambda) = b_{11}\lambda_1 + b_{12}\lambda_2 + c_1 \lambda_1^2 + \int_{\mathbb{R}_+^2} (\e^{-z_1\lambda_1} - 1 + z_1\lambda_1)m_1(\d z),
 \end{equation}
 \begin{equation}\label{tilphi 2'}
 \tilde\phi_2(\lambda) = b_{21}\lambda_1 + b_{22}\lambda_2 + c_2 \lambda_2^2 + \int_{\mathbb{R}_+^2} (\e^{-z_2\lambda_2} - 1 + z_2\lambda_2)m_2(\d z).
 \end{equation}
 Similarly, there exists~$r\mapsto \tilde U_{r,t}(\xi,\lambda)\in\mathbb{R}_+^2$~as the unique positive strong solution to
 \begin{equation}\label{u}
   \tilde U_{r,t}^{(i)}(\xi,\lambda) = \lambda_i - \int_r^t \e^{\xi(s)}\tilde\phi_i(e^{-\xi(s)}\tilde U_{s,t}(\xi,\lambda))\d s,\quad i=1,2,\quad r\in I\cap (-\infty,t].
 \end{equation}

 The\emph{ local projection} of $\phi$ defined by \eqref{phi 1'}-\eqref{phi 2'} is the function $\phi^*$ from $\mathbb R_+$ to $\mathbb R_+^2$ defined by
  \begin{equation}\label{staphi 1'}
\phi_1^*(x) = (b_{11} + b_{12})x + c_1 x^2 + \int_{\mathbb{R}_+^2} (\e^{-x z_1} - 1 + x z_1)m_1(\d z),
 \end{equation}
 \begin{equation}\label{staphi 2'}
\phi_2^*(x) = (b_{21} + b_{22})x + c_2 x^2 + \int_{\mathbb{R}_+^2} (\e^{-x z_2} - 1 + x z_2)m_2(\d z).
 \end{equation}

  \begin{cond}\label{cond}
 For $x=1,2,$ $\phi^*_x\geq \varphi(z),$ where $\varphi$ is a branching mechanism of a single-type continuous-state continuous-time branching process in L\'evy random environment satisfying $\int^{\infty}\varphi(z)^{-1}\d z<\infty.$ The corresponding random cumulant semigroup is $w_{r,t}$ given by
 $$ w_{r, t}(\xi, \lambda)=\lambda-\int_{r}^{t} \mathrm{e}^{\xi(s)} \varphi(e^{-\xi(s)} w_{s, t}(\xi, \lambda)) \d s,\quad r \in I \cap(-\infty, t].$$
 \end{cond}
 \begin{thm}\label{0725}
 Suppose that Condition \ref{cond} is satisfied. Then
 $$ u_{r,t}^{(i)}(\xi,\lambda)\leq w_{r,t}(\xi,\|\lambda\|), \quad\textit{for any}\ i=1,2,\lambda\in\mathbb{R}_+^2. $$ Furthermore, if  $\liminf\limits_{t\rightarrow\infty}\xi(t)=-\infty,$ then $\lim\limits_{t\rightarrow\infty}\|\bar{v}_{0, t}^{\xi}\|=0,\mathbf{P}-a.s.,$ where $u_{r,t}$ and $v_{r,t}$ are defined by \eqref{191006u} and \eqref{191006v}, $\bar{v}_{0, t}^{\xi}:=\lim\limits_{\lambda\rightarrow\infty}v_{0,t}(\xi,\lambda).$
 \end{thm}
 \begin{proof}
 By Proposition 2.9 in \cite{li}, equation\ \eqref{u} can be rewritten into
 \begin{equation*}\label{u12'}
   \tilde U_{r,t}^{(1)}(\xi,\lambda) = \e^{b_{12}(t-r)}\lambda_1 - \int_r^t \e^{b_{12}(s-r)}[\e^{\xi(s)}\tilde\phi_1(e^{-\xi(s)}\tilde U_{s,t}(\xi,\lambda))+b_{12}\tilde U_{s,t}^{(1)}(\xi,\lambda)]\d s,
 \end{equation*}
\begin{equation*}\label{u2'}
   \tilde U_{r,t}^{(2)}(\xi,\lambda) = \e^{b_{21}(t-r)}\lambda_2 - \int_r^t \e^{b_{21}(s-r)}[\e^{\xi(s)}\tilde\phi_2(e^{-\xi(s)}\tilde U_{s,t}(\xi,\lambda))+b_{21}\tilde U_{s,t}^{(2)}(\xi,\lambda)]\d s,
 \end{equation*}
   It is clear that
   \begin{equation*}\label{s1}
\tilde\phi_1(\lambda)=\phi_1^*(\lambda_1) - b_{12}\lambda_1 +b_{12}\lambda_2,
 \end{equation*}
 \begin{equation*}\label{s2}
\tilde\phi_2(\lambda)=\phi_2^*(\lambda_2) - b_{21}\lambda_2 +b_{21}\lambda_1.
 \end{equation*}
 Hence,
 \begin{eqnarray*}\label{u12'}
 \tilde U_{r,t}^{(1)}(\xi,\lambda) \ar=\ar \e^{b_{12}(t-r)}\lambda_1 - \int_r^t \e^{b_{12}(s-r)}[\e^{\xi(s)}\phi^*_1(e^{-\xi(s)}\tilde U^{(1)}_{s,t}(\xi,\lambda))+b_{12}\tilde U_{s,t}^{(2)}(\xi,\lambda)]\d s,\cr
   \tilde U_{r,t}^{(2)}(\xi,\lambda) \ar=\ar \e^{b_{21}(t-r)}\lambda_2 - \int_r^t \e^{b_{21}(s-r)}[\e^{\xi(s)}\phi^*_2(e^{-\xi(s)}\tilde U^{(2)}_{s,t}(\xi,\lambda))+b_{21}\tilde U_{s,t}^{(1)}(\xi,\lambda)]\d s.
 \end{eqnarray*}
Define $$P^g_{r,t}f(1)= \e^{b_{12}(t-r)}f(1),P^g_{r,t}f(2)= \e^{b_{21}(t-r)}f(2),$$ and use the relation $\lambda_i=f(i),i=1,2.,$ the above equations can be rewritten into
%  \begin{eqnarray}\label{u1'}
% \tilde U_{r,t}^{(1)}(\xi,\lambda)= P_{r,t}^{g} f(1) - \int_r^tP_{r,s}^{g}[\e^{\xi(s)}\tilde\phi(e^{-\xi(s)}\tilde U_{s,t}(\xi,\lambda))+b_{12}\tilde U_{s,t}(\xi,\lambda)](1)\d s,i=1,2.
% \end{eqnarray}
  \begin{eqnarray}\label{u11'}
   \tilde U_{r,t}^{(i)}(\xi,\lambda) \ar=\ar P_{r,t}^{g} f(i)- \int_r^t \d s\int_{\{1,2\}}[\e^{\xi(s)}\phi^*_1(e^{-\xi(s)}\tilde U^{(1)}_{s,t}(\xi,\lambda)) \cr
\ar\qquad\ar\qquad\qquad\qquad\qquad\qquad\qquad-g(i,\tilde U_{s,t}(\xi,\lambda))]P_{s,t}^g(i,\d y),
 \end{eqnarray}
Let $(\tilde P_{r,t}:t\geq r \geq 0)$ be an inhomogeneous Borel right semigroup defined  by
\begin{equation}\label{til}
\tilde{P}_{r,t} f(x)=P_{r,t}^{g} f(x)+\int_{r}^{t} \d s \int_{E} g(y,\tilde{P}_{s,t} f) P_{r,s}^{g}(x, \d y),
\end{equation}
where $g(s,y, f)$ is defined by \eqref{g}. We claim that the unique solution of\  \eqref{til} is given by
 \begin{equation}\label{duo1}
 \tilde{P}_{r,t} f(x)=P_{r,t}^{g} f(x)+\sum_{k=1}^{\infty} \int_{r}^{t} \d s_{1} \int_{s_{1}}^{t} \cdots \int_{s_{k-1}}^{t} P_{r, s_1}^{g} g P_{s_1, s_2}^{j} \cdots g P_{s_{k-1}, s_k}^{g}f(x) \d s_{k}.
 \end{equation}
Indeed, the $i$-th term of the series in\ \eqref{duo1} is bounded by $K_{r,t},$ where $K_{r,t}:=-[b_{12}(t-r)\e^{b_{12}(t-r)}\wedge b_{21}(t-r)\e^{b_{21}(t-r)}].$  Then the series converges uniformly on $[0,\infty)\times\{1,2\}.$ For the uniqueness of the solution, suppose that $(r,x)\mapsto z_{r,t}(x)$ is a locally bounded solution of\ \eqref{til} with $ z_{r,t}(x)\equiv 0$,
$$\| z_{r,t}\|\le K_{r,t} \sup_{r\le s\le t}\| z_{s,t}\|.$$
Notice that $K< 1,$ we have $\| z_{r,t}\|=0$ for $0\le r\le t.$ That gives the uniqueness and hence the unique solution of\  \eqref{til} is given by \eqref{duo1}.

 From equation \eqref{til} and equation \eqref{u11'} it follows that,
    \begin{eqnarray*}\label{uu}
  \tilde U_{r,t}(\xi,f)(1)\ar:= \ar \tilde U_{r,t}^{(i)}(\xi,f) \cr
  \ar=\ar \tilde P_{r,t}f(1) - \int_r^t \int_{\{1,2\}}[\e^{\xi(s)}\phi^*_1(e^{-\xi(s)}\tilde U^{(1)}_{s,t}(\xi,\lambda))\cr
\ar\quad\ar\qquad\qquad\qquad\qquad\qquad\qquad\qquad\qquad-g(i,\tilde U_{s,t}(\xi,\lambda))]P_{r,s}^g(i,\d y)\d s,
 \end{eqnarray*}
 Using the above relation successively,
 \begin{eqnarray}
\ar\ar\tilde U_{r,t}(\xi,f)(1)\cr
\ar=\ar\tilde{P}_{r, t} f(1)-\int_{r}^{t} P_{r,s_1}^{g}[e^{\xi(s_1)} \phi_{1}^{*}(e^{-\xi(s_1)} \tilde{U}_{s_1, t} f(1))] \d s_{1} \cr
\ar\quad\ar\qquad\qquad\qquad\qquad\qquad\qquad-\int_{r}^{t} \d s_{1} \int_{s_{1}}^{t} P_{r,s_1}^{g} g P_{s_1,s_2}^{g} g[\tilde{P}_{s_2, t} f-\tilde{U}_{s_2, t} f] \d s_{2}\cr
\ar=\ar\tilde{P}_{r, t} f(1)-\int_{r}^{t} P_{r,s_1}^{g}[e^{\xi(s_1)} \phi_{1}^{*}(e^{-\xi(s_1)} \tilde{U}_{s_1, t} f(1))] \d s_{1} \cr
\ar\quad\ar-\int_{r}^{t} \d s_{1} \int_{s_{1}}^{t} P_{r,s_1}^{g} g P_{s_1,s_2}^{g}[e^{\xi(s_{2})} \phi_{1}^{*}(e^{-\xi(s_{2})} \widetilde{U}_{s_{2}, t} f(1))] \d s_{2}\cr
\ar\quad\ar-\sum_{k=3}^{n} \int_{r}^{t} \d s_{1} \int_{s_{1}}^{t} \cdots \int_{s_{k-1}}^{t} P_{r,s_1}^{g} g P_{s_1,s_2}^{j} \cdots g P_{s_{k-1}, s_k}^{g}[e^{\xi(\xi_{k})} \phi_{1}^{*}(e^{-\xi(s_{k})} \tilde{U}_{s_{k}, t} f(1))] \d s_{k}\cr
\ar\quad\ar+\varepsilon_n(r,t,1),
 \end{eqnarray}
 where
 $$\varepsilon_{n}(r, t, 1)=\int_{r}^{t} \d s_{1} \int_{s_{1}}^{t} \cdots \int_{s_{k-1}}^{t} P_{r,s_1}^{g} g P_{s_1,s_{2}}^{g} \ldots g P_{s_{k-1},s_k}^{g} g[\tilde{P}_{s_{k}, t} f-\tilde{U}_{s_{k}, t} f] \d s_{k}.$$
 By Proposition 5.1 in \cite{qinzheng}, there exist constants $C_1>0,C_2>0,$ such that $\|\tilde U_{r,t} f\|\le C_1\| f\|\e^{-C_2(t-r)}.$
 Therefore,
 \begin{eqnarray*}
\left\|\varepsilon_{n}(r,t, \cdot)\right\| \ar \leq\ar(1+C_1\mathrm{e}^{-C_2( t-r)})\| f\|\| g(\cdot, 1)\|^{n} \int_{r}^{t} \d s_{1} \int_{{s_{1}}}^t \d s_{2} \cdots \int_{s_{n-1}}^t \d s_{n} \cr
\ar \leq\ar(1+C_1\mathrm{e}^{-C_2( t-r)})\|f\|\|g(\cdot, 1)\|^{n} \frac{(t-r)^{n}}{n !}.
 \end{eqnarray*}
Letting $n\rightarrow\infty$ and using \eqref{duo1} we get,
 $$\tilde{U}_{r,t} f(x)=\tilde{P}_{r,t} f(x)-\int_r^t\d s\int_{\{1,2\}}\mathrm{e}^{\xi(s)}\phi^*_x(y,\mathrm{e}^{-\xi(s)}\tilde U_{s,t}f(y))\tilde{P}_{s,t} (x,\d y),x=1,2. $$
 Hence, we may see $(\tilde U_{r,t})_{t\geq r\geq 0}$ as the cumulant semigroup of a Dawson-Watanabe superprocess with branching mechanism $\phi^*$ and underlying transition semigroup $(\tilde P_{r,t})_{t\geq r\geq 0}.$
By Proposition \ref{com} we have,
$$ \tilde U_{r,t}f(x)\leq \tilde U_{r,t}\|f\|(x)\leq w_{r,t}(\| f\|).$$
Similary, $u_{r,t}^{(i)}(\xi,\lambda)\le\tilde U_{r,t}^{(i)}(\xi,\lambda)$ for $i=1,2,\lambda\in\mathbb{R}_+^2,$ since $\tilde\phi_i(\lambda)\le\phi_i(\lambda).$

Furthermore, if  $\liminf\limits_{t\rightarrow\infty}\xi(t)=-\infty,$ according to Corollary 4.4 in \cite{helixu}, we have $\lim\limits_{t\rightarrow\infty}\lim\limits_{\lambda\rightarrow\infty}w_{r,t}(\lambda)=0, \mathbf{P}\text{-}a.s.$ Thus, $\lim\limits_{t\rightarrow\infty}\|\bar{v}_{0, t}^{\xi}\|=0,\mathbf{P}\text{-}a.s.$

\end{proof}

 We are now in a position to prove our second main result.

\noindent\textbf{Proof of Theorem 2.2} Because $\p[\xi(1)]<{ \beta}<\frac{1}{2}(tr({b})-\sqrt{{ \Delta}}),$ and $\frac{1}{2}(tr({b})-\sqrt{{ \Delta}})$ is the smallest eigenvalue of $b,$ we may adjust the parameters so that $\p[\xi(1)]<0,$ and the eigenvalues of $b$ are strictly positive. Thus, $\xi(t)\rightarrow\-\infty,$ as $t\rightarrow\infty.$ By Theorem \ref{0725}, $v_{r,t}(\xi,\lambda)<\infty, \mathbf{P}$-a.s. Thus, for any bounded Borel function $f$ on $\mathbb{R}_+^2$ and any $t\ge 0,$
\begin{eqnarray}
\vert\bar{Q}_tf(x)-\bar{Q}_tf(y)\vert\ar\le\ar\mathbf{P}\vert f(X_t(x))-f(X_t(y))\vert\mathbf{1}_{\{T(x,y)>t\}}\cr
\ar\le\ar 2\|f\|\mathbf{P}[T(x,y)>t]\cr
\ar=\ar 2\|f\|\mathbf{P}[1-\exp\{-\langle (y-x),\bar{v}_{0,t}(\xi)\rangle\}]\cr
\ar\le\ar 2\|f\|\mathbf{P}\mathrm{min}[1,\vert\langle (y-x),\bar{v}_{0,t}(\xi)\rangle\vert].
\end{eqnarray}
 Let $H$ be any coupling of $(\delta_x,\mu)$. By convexity of the Wasserstein distance, we obtain

\begin{eqnarray}
\|\delta_x\bar Q_t-\mu\|_{TV} \ar \leq\ar \int_{\mathbb{R}_+^2 }\|\bar{Q}_t(x,\cdot)-\bar{Q}_t(y,\cdot)\|_{TV} H(\d x, \d y) \cr
\ar\leq\ar 2\int_{\mathbb{R}_+^2 }\mathbf{P}[1-\exp\{-\langle (y-x),\bar{v}_{0,t}(\xi)\rangle\}]H(\d x, \d y).%\cr
%\ar \leq\ar 2 \int_{\mathbb{R}_+^2 } \mathbf{P}[\min \{1,\vertx-y\vert \vert\bar{v}_{0, t}^{\xi}\vert\}] H(\d x, \d y)
\end{eqnarray}
Since $\liminf\limits_{t\rightarrow\infty}\xi(t)=-\infty,$ $\lim\limits_{t\rightarrow\infty}\|\bar{v}_{0, t}^{\xi}\|=0$ according to Theorem \ref{0725}. By dominated convergence,
$$\lim\limits_{t\rightarrow\infty}\|\delta_x\bar Q_t-\mu\|_{TV}=0.$$
By choosing $H$ as the optimal coupling of $(\delta_x,\mu)$ with respect to $W_{1}$ and using
$$
\int_{\mathbb{R}_+^2} \mathbb{P}[\mathrm{min} \{1,\vert x-y\vert\| \bar{v}_{0, t}^{\xi}\|\}] H(\d x, \d y) \leq \mathbb{P}[\|\bar{v}_{0, t}^{\xi}\|] \int_{\mathbb{R}_+^2}\vert x-y\vert H(\d x, \d y),
$$
we finish the proof.\qed
%%%%%%%%%%%%%%%%%%%%%%%%%%%%%%%%%%%%%%%%%%%%%%%%%%%%%%%%%%%%%%%%%%%%%%%%%%%%%%%%%%%%%%%%%%%%%%%%%%%%%%%%%%%%%%%%%%%%%%%%%%%%%%%%%%%%
%%%%%%%%%%%%%%%%%%%%%%%%%%%%%%%%%%%%%%%%%%%%%%%%%%%%%%%%%%%%%%%%%%%%%%%%%%%%%%%%%%%%%%%%%%%%%%%%%%%%%%%%%%%%%%%%%%%%%%%%%%%%%%%%%%%%
\bmhead{Acknowledgments}

The authors are greatly indebted to Prof. Zenghu Li for his valuable suggestions and for the guidance over the past years. 
\section*{Declarations}
\begin{itemize}
\item This work was supported by the National Natural Science Foundation of China (No. 11531001) and the Education and Scientific Research Project for Young and Middle-aged Teachers in Fujian Province of China (No. JAT200072).
\item The authors have no conflicts of interest to declare. All co-authors have seen and agree with the contents of the manuscript.
\item  The data that support the findings of this study are available on request from the corresponding author.
\end{itemize}
%%%%%%%%%%%%%%%%%%%%%%%%%%%%%%%%%%%%%%%%%%%%%%%%%%%%%%%%%%%%%%%%%%%%%%%%%%%%%%%%%%%%%%%%%%%%%%%%%%%%%%%%%%%%%%%%%%%%%%%%%%%%%%%%%%%%
%%%%%%%%%%%%%%%%%%%%%%%%%%%%%%%%%%%%%%%%%%%%%%%%%%%%%%%%%%%%%%%%%%%%%%%%%%%%%%%%%%%%%%%%%%%%%%%%%%%%%%%%%%%%%%%%%%%%%%%%%%%%%%%%%%%%

\end{document}